\documentclass[12pt]{amsart}

\usepackage{
amsfonts,
latexsym,
amssymb,
enumerate,
verbatim,
mathrsfs,
}

\newcommand{\labbel}{\label}

\newcommand{\bm}[1] { \, ${\boldmath$#1$}$} 
\newcommand{\bmw}[1] {${\boldmath$#1$}$}

\newtheorem{theorem}{Theorem}[section]
\newtheorem{lemma}[theorem]{Lemma}

\newtheorem{proposition}[theorem]{Proposition} 
 
\newtheorem{corollary}[theorem]{Corollary} 
 
\newtheorem*{claim}{Claim}

\newtheorem*{theorem*}{Theorem}
\newtheorem*{corollary*}{Corollary}

\theoremstyle{definition}

\newtheorem*{JDSP}{The J{\'o}nsson distributivity spectrum problem}
\newtheorem*{JDSPG}
{The generalized J{\'o}nsson distributivity spectrum problem}

\theoremstyle{remark}

\newtheorem*{acknowledgement}{Acknowledgement}

 \allowdisplaybreaks[1]

\begin{document}
 
\title[J{\'o}nsson distributivity spectrum]
{On the  J{\'o}nsson  distributivity spectrum}

\author{Paolo Lipparini} 
\address{ 
Dipartimento Distributivo di Matematica\\Viale della  Ricerca
 Scientifica\\Universit\`a di Roma ``Tor Vergata'' 
\\I-00133 ROME ITALY}
\urladdr{http://www.mat.uniroma2.it/\textasciitilde lipparin}


\thanks{Work performed under the auspices of G.N.S.A.G.A. Work 
partially supported by PRIN 2012 ``Logica, Modelli e Insiemi''}

\begin{abstract}
Suppose throughout that $\mathcal V$ is a
 congruence distributive variety. If
$m \geq 1$, let $ J _{ \mathcal V} (m) $ be the smallest natural
number $k$ such that the congruence identity 
 $\alpha ( \beta \circ \gamma \circ \beta \dots ) \subseteq \alpha \beta 
\circ  \alpha \gamma \circ \alpha \beta \circ \dots $
holds in $\mathcal V$, with $m$ occurrences of  $ \circ$ 
on the left  and $k$ occurrences of $\circ$ on the  
right.
We show that if $ J _{ \mathcal V} (m) =k$, then 
$ J _{ \mathcal V} (m \ell )
\leq  k \ell $, for every natural number $\ell$.   
The key to the proof is  an identity
which, through a variety,
 is equivalent to the above congruence identity,
but involves also reflexive and admissible relations. 
If $ J _{ \mathcal V} (1)=2 $, that is, $\mathcal V$ 
is $3$-distributive, then $ J _{ \mathcal V} (m) \leq m $,
for every $m \geq 3$ (actually, a more general result is presented
which holds even in nondistributive varieties).   
If $\mathcal V$ is $m$-modular, that is, congruence modularity 
of $\mathcal V$ is witnessed by $m+1$ Day terms, then
$  J _{ \mathcal V} (2) 
\leq  J   _{ \mathcal V} (1) + 2m^2-2m -1 $.
Various  problems are stated at various places.
\end{abstract}

\maketitle

\section{The J{\'o}nsson distributivity spectrum} \labbel{sp} 

Obviously, an algebra $\mathbf A$ is congruence distributive if and only if,
for every natural number $m \geq 2$, 
the congruence identity
$ 
\alpha ( \beta \circ _m \gamma ) \subseteq \alpha \beta + \alpha \gamma 
$ 
holds in $\mathbf {Con} (\mathbf A)$ 
(more precisely, in the algebra of reflexive and admissible relations on $\mathbf A$).
Here $\alpha$, $\beta$, \dots \ 
are intended to vary among congruences of $\mathbf A$,
juxtaposition denotes intersection,
$+$ is join in the congruence lattice  and
$ \beta \circ _m \gamma $  denotes
$ \beta \circ \gamma \circ \beta \dots$
with $m$ factors, that is, 
with $m-1$ occurrences of $\circ$.

Let us say that a congruence  identity
holds in some variety $\mathcal V$ if
it holds in every algebra in $\mathcal V$.
By a celebrated theorem by J{\'o}nsson \cite{JD}, 
a milestone both in the theory  
 of Maltsev conditions and 
in the theory
of congruence distributive varieties,
a variety $\mathcal V$ is congruence distributive if and only if 
there is some $n$ such that the congruence identity
\begin{equation}\labbel{d2}    
 \alpha ( \beta \circ \gamma  ) \subseteq \alpha \beta 
\circ _n \alpha \gamma 
  \end{equation} 
holds in $\mathcal V$. 
In other words, for varieties, 
taking $m=2$ in the above paragraph is already enough. 
J{\'o}nsson actual statement in \cite{JD} is about a set 
of  terms naturally arising from
identity \eqref{d2}, rather than about the
identity itself. J{\'o}nsson terms shall be recalled later.
Stating results with regard to congruence  identities
rather than terms
is simpler and easier to understand,
while proofs usually require the corresponding terms. 
 Compare the perspicuous  discussion in  
 Tschantz \cite{T}.

 J{\'o}nsson proof in \cite{JD} goes on by showing that
if some variety $\mathcal V$ has  terms
witnessing  \eqref{d2}, then, 
for every $m$, the inclusion
$ 
\alpha ( \beta \circ _m \gamma ) \subseteq \alpha \beta + \alpha \gamma 
$ 
 holds in $\mathcal V$.
(By the way, let us mention that J{\'o}nsson paper \cite{JD} 
contains a big deal of other fundamental results about distributive 
varieties with significant and unexpected applications to lattices,
among other.)
It follows easily from J{\'o}nsson  proof that, for every $m$, there is some
$k$ (which depends only on $m$ and  on the $n$
given by \eqref{d2}, but otherwise not on  the variety) such that 
\begin{equation}\labbel{d3}  
\tag*{$(m,k)$-dist}
 \alpha ( \beta \circ _{m} \gamma ) \subseteq 
\alpha \beta \circ _{k} \alpha \gamma 
\end{equation}      

A variety is $\Delta_{k}$
 in the sense of  \cite{JD} if and only if  it
satisfies
 $(2,k)$-dist. 
Such varieties are sometimes called 
\emph{$k$-distributive}, or are said to \emph{have 
$k+1$ J{\'o}nsson terms}.
If $k$ is minimal with the above property (with $m=2$),
$\mathcal V$ is said to be of \emph{J{\'o}nsson level} $k$
in Freese and Valeriote \cite{FV}.

If we slightly modify the  proof of J{\'o}nsson theorem as presented
in Burris and Sankappanavar \cite[Theorem 12.6]{BS} or in  McKenzie,  McNulty, 
and Taylor \cite[Theorem 4.144]{MMT},
 we see that 
if a variety $\mathcal V$ satisfies $(2,k+1)$-dist, then 
$\mathcal V$ satisfies $( \ell +1 , k \ell +1 )$-dist,
for every $ \ell \geq 1$. 
This result is also a special case of Corollary \ref{ell} below.
To formulate this and other results in a more concise way,
 it is natural to introduce the following 
 \emph{J{\'o}nsson 
distributivity function} $ J _{ \mathcal V} $
of a congruence distributive variety $\mathcal V$.
For  every positive natural number $m$, 
we set  $ J _{ \mathcal V} (m) $ 
to be the least $k$ such that 
$\mathcal V$ satisfies the identity $(m+1,k+1)$-dist.
The ``shift by $1$'' in the above notation
 will greatly simplify subsequent statements. 
For example, the above remark is more neatly 
stated by asserting that if $ J _{ \mathcal V} (1) = k $,
then $ J _{ \mathcal V} (\ell) \leq k \ell $, for every positive $\ell $.
We can now ask the following problem.

\begin{JDSP}
 \labbel{prob1}
\emph{Which functions (with domain
the set of positive natural numbers) can be realized as
 $J _{ \mathcal V} $,
for some congruence distributive variety  $\mathcal V$?}
 \end{JDSP}

Obviously, $J _{ \mathcal V} $ is a monotone function.
By the above comments,
if  $ J _{ \mathcal V} (1) = 1$,
then  $ J _{ \mathcal V} (m) \leq m$,
for every positive $m $.
Moreover if, for some $k$,  $ J _{ \mathcal V} (k) < k$,
 then $\mathcal V$ is   
$k$-permutable: just take $ \alpha =1$ in equation $(m,k)$-dist. 
If $\mathcal V$ is   
$k$-permutable, then 
$ J _{ \mathcal V} (m) < k$,
for every $m$. 

As a consequence of the above observations, if $ J _{ \mathcal V} (1) = 1$,
then
$ J _{ \mathcal V} $ is either the identity function,
or is the identity up to some point and then it is a constant function.
An example of the first eventuality is the variety of lattices;
on the other hand, 
in the variety of 
\emph{$n$-Boolean algebras} from  Schmidt
\cite{S} we have
  $ J _{ \mathcal V} (m) = \min \{ m, n\} $.
Indeed, $n$-Boolean algebras have a lattice operation,
  are $n+1$-permutable but, in general, not $n$-permutable;
see also Hagemann and Mitschke \cite{HM} and \cite[Example 2.8]{cd}. 
For convenience, we shall use J{\'o}nsson paper \cite{cd} 
as a reference for this and other examples. The author 
believes that this is the appropriate place to mention 
that J{\'o}nsson \cite{cd} 
has had a profound influence in his mathematical formation.

 Mitschke \cite{M} shows that the variety 
of \emph{implication algebras}
is $3$-permutable, not permutable, $\Delta_3$ and not $\Delta_2$.
See also \cite[Example 1]{HM} and \cite[Example 2.6]{cd}. 
Hence in the variety of implication algebras
we have
$ J _{ \mathcal V} (m) = 2 $,
for every $m $.
Freese and Valeriote \cite{FV},
using the reduct of an algebra formerly constructed by 
 Kearnes \cite{K},
show that, for every $n$,  there is an
 $n$-permutable variety which is $\Delta_{n}$ and  not $\Delta_{n-1}$.
See \cite[p. 70--71]{FV}.   
Thus $ J _{ \mathcal V} $
is constantly $n-1$ in this variety.

 We now observe that the set of those functions
which can be represented as $ J _{ \mathcal V} $,
for some variety, is closed under pointwise maximum.
This is immediate from the result that 
the non-indexed product of    two varieties
$\mathcal V$ and $\mathcal V'$
satisfies exactly the same strong Maltsev conditions
satisfied both by $\mathcal V$ and $\mathcal V'$.
See  Neumann \cite{N}, Taylor \cite{Ta} or \cite[p. 368--369]{cd}.
 We also need the easy fact
that, for every $m$ and $k$, the  condition
$ J _{ \mathcal V} (m) \leq k$ is equivalent to a strong Maltsev 
condition; for example, this is a consequence of the equivalence of (A) and (B)
in Theorem \ref{smile} below.
 
\begin{proposition} \labbel{nip}
If $\mathcal V$ and $\mathcal V'$
are congruence distributive varieties,
then their non-indexed product $\mathcal V''$ 
is such that 
$ J _{ \mathcal V''} (m) = \allowbreak 
\max  \allowbreak \{  J _{ \mathcal V} (m) , \allowbreak  J _{ \mathcal V'} (m)  \}  $,
for every positive natural number $m$. 
 \end{proposition}

We do not know whether,  for every pair  $\mathcal V$, $\mathcal V'$,
we always have some $\mathcal V''$ such that   
$ J _{ \mathcal V''} (m) = \allowbreak 
\min  \allowbreak \{  J _{ \mathcal V} (m) , \allowbreak  J _{ \mathcal V'} (m)  \}  $,
for every  $m$. 

If in Proposition \ref{nip} we consider the variety of lattices 
and the mentioned variety 
from \cite[p. 70--71]{FV}, we get 
$ J _{ \mathcal V''} (m) = \max \{ n-1, m\}$.
By taking the non-indexed product of the variety of 
$n'$-Boolean algebras and again the variety 
from \cite[p. 70--71]{FV}, with $ n \leq n'$,
we have  $ J _{ \mathcal V''} (m) = n-1$, for $m \leq n-1$
and 
  $ J _{ \mathcal V''} (m) = \min \{ m, n'\} $,
for $m > n-1$. 

The above examples suggest
that $ J _{ \mathcal V}(m) $
has little influence on the values
of $ J _{ \mathcal V} (m')$,
for $m' < m$.
On the other hand, we are going to show 
that $ J _{ \mathcal V}(m) $
puts some quite restrictive bounds on 
 $ J _{ \mathcal V} (m')$,
for $m' > m$, as we already mentioned for the easier case
$m=1$.

\section{Bounds on higher levels of the spectrum} \labbel{bo} 

Let $R$, $S$,  \dots \ be variables intended to be interpreted as
reflexive and admissible (binary) relations on some algebra.
If  $R$ is such a relation, let $R ^\smallsmile $ denote the \emph{converse}
of $R$, that is, $b \mathrel {R ^\smallsmile }  a $
 if and only if $a \mathrel R b  $.
In the next theorem   we show that, for a variety,
the congruence  identity $(m+1,k+1)$-dist   is equivalent to 
the relation identity $ \alpha (R \circ _{m} R ^\smallsmile  ) \subseteq
\alpha R \circ _{k} \alpha R ^\smallsmile $,
with the further provision that if $R$ can be expressed
as a composition, then $\alpha R$ 
and $\alpha R ^\smallsmile $
 factor out. See   
condition (3) in the next theorem  for a formal statement.
The latter provision is necessary, since, without it,
the case $m=1 \leq k$ would be trivially true 
in every variety and, 
for every $m > 1$
and using Day terms, it can be shown
that a variety is congruence modular if and only if there is 
some $k$ such that   
the relation identity 
$ \alpha (R \circ _{m} R ^\smallsmile  ) \subseteq
\alpha R \circ _{k} \alpha R ^\smallsmile $ holds.
Hence this identity alone is too weak for our purposes.
See \cite{ricm}. 

Recall that a \emph{tolerance} is a symmetric and reflexive
admissible relations. We shall prove Part (C) in the next theorem 
in the general case when $\alpha$ is a tolerance, rather than a congruence.
In particular, we get that, through a variety, 
the identity $(m,k)$-dist is equivalent to the same identity in which
$\alpha$ is only assumed to be a tolerance.
However this stronger version shall not be
used  in what follows, hence the reader might
 always assume to be in the simpler case 
in which  $\alpha$ is a congruence.

\begin{theorem} \labbel{smile}
For every variety $\mathcal V$
and integers $m, k \geq 1$,
 the following conditions are equivalent.
 \begin{enumerate}[(A)]  
 \item 
 $  J _{ \mathcal V} (m)  \leq k $,  that is, 
$\mathcal V$  satisfies the congruence identity $(m+1,k+1)$-dist
\begin{equation*}\labbel{ucac}      
 \alpha ( \beta \circ _{m+1} \gamma )
 \subseteq \alpha \beta \circ_{k+1} \alpha \gamma 
\end{equation*}  
(equivalently, we can ask that the free algebra in $\mathcal V$ generated by
$m+2$ elements satisfies the above identity.)
\item
$\mathcal V$ has $m+2$-ary  terms
$t_0, \dots, t_{k+1}$ such that the following identities hold in $\mathcal V$:
\begin{gather} \labbel{buh} 
\tag{B1}
  x  = t_0(x,x_1,x_2, x_3, \dots, x_{m}, x_{m+1}),    \\
\labbel{buhh}
\tag{B2}
  x =t_i(x,x_1,x_2, x_3, \dots, x_{m}, x), \quad\quad 
  \text{ for } 
0 \leq i \leq k+1, \\ 
 \labbel{aa} \tag{B3} 
\left \{
\begin{split}  
 t_{i}(x_0,x_0,x_2, x_2,x_4, x_4, \dots ) &=
t_{i+1}(x_0,x_0,x_2, x_2,x_4, x_4, \dots ), \\ 
  & \quad \quad\quad \quad\quad \quad \text{ for even $i$,\ } 
0 \leq i \leq k,
  \\ 
  t_{i}(x_0,x_1,x_1, x_3,x_3, \dots )
&=t_{i+1}(x_0,x_1,x_1, x_3,x_3, \dots ), \\ 
 &\quad\quad\quad \quad\quad \quad \text{ for odd $i$,\ } 
0 \leq i \leq k,
 \end{split}
\right .  \\ 
\labbel{uff}
\tag{B4}
 t_{k+1}(x_0,x_1,x_2, x_3, \dots, x_{m}, z)=z  
\end{gather}   
\item
For every algebra $\mathbf A \in \mathcal V$,
every positive integer $\ell$
(equivalently, for $\ell=1$),
every tolerance $\alpha$ of $\mathbf A$
and all  reflexive and admissible relations
$R$, $S_0$, \dots, $S_ \ell $ on $\mathbf A$, 
if $R= S_0 \circ S_1 \circ  \dots  \circ S_ \ell$
and    $ \Theta = \alpha S_0 \circ \alpha S_1 \circ  \dots  \circ \alpha S_ \ell$, then
\begin{equation}\labbel{uuu}
\tag{C1}
\alpha (R \circ_{m} R ^\smallsmile ) \subseteq 
\Theta \circ_ {k} \Theta ^\smallsmile 
  \end{equation}    
\end{enumerate} 
 \end{theorem}  

\begin{proof}
The equivalence of (A) and (B) is an instance of the
Pixley-Wille algorithm \cite{P,W}  (actually, it can be seen as a good exercise
to check a student's understanding of the algorithm).
We shall need here only (A) $\Rightarrow $  (B), which can be proved as follows.
Consider the free algebra $\mathbf F_{\mathcal V}(m+2)$ in $\mathcal V$ over 
$m+2$ generators $y_0, \dots, y_{m+1}$ and let $\alpha$ 
be the congruence generated by $(y_0,y_{m+1})$,
$\beta$ be the congruence generated by 
$\{ (y_0, y_1), (y_2, y_3), (y_4, y_5), \dots \}$    
and 
$ \gamma $ be the congruence generated by 
$\{ (y_1, y_2), (y_3, y_4), \dots \}$.
Thus 
$(y_0, y_{m+1}) \allowbreak \in \alpha ( \beta \circ _{m+1} \gamma ) $     
hence, by (A), 
$ (y_0, y_{m+1}) \in \alpha \beta \circ_{k+1} \alpha \gamma $.
This latter relation is witnessed by 
$k+2$ elements  of $\mathbf F_{\mathcal V}(m+2)$ 
which give rise to terms witnessing (B).

(B) $\Rightarrow $  (C) 
Let $(a,c) \in \alpha (R \circ_{m} R ^\smallsmile )$
in some algebra $\mathbf A \in \mathcal V$.
 This is witnessed by elements $a=b_0, \allowbreak  b_1, \allowbreak 
 b_2, \dots, b_{m}= c $
such that  $b_0 \mathrel R  b_1 \mathrel { R ^\smallsmile }  b_2 \mathrel R
b_3 \mathrel { R ^\smallsmile }  b_4 \dots$ \ 
Furthermore 
$a \mathrel \alpha  c$.  

First suppose that $m$ is even. 
For $0 \leq i \leq k$, we shall consider the elements 
\begin{align*} 
e_i = \ &  t_{i}(b_0, b_0,b_2, b_2, \dots, b_{m-2}, b_{m-2}, b_{m}, b_{m}) =
\\
&
t_{i+1}(b_0, b_0,b_2, b_2, \dots , b_{m-2}, b_{m-2}, b_{m}, b_{m}) 
\quad \text{ for $i$ even,}
\\
e_i = \  &  t_{i}(b_0, b_1,b_1, b_3, \dots, b_{m-3}, b_{m-1}, b_{m-1}, b_{m}) =
\\
&
t_{i+1}(b_0, b_1,b_1, b_3, \dots, b_{m-3}, b_{m-1}, b_{m-1}, b_{m}) 
\quad \text{ for $i$ odd,}
 \end{align*} 
where the identities follow from \eqref{aa}.
In writing the above formula we are supposing that $m$
is large enough; otherwise, say, for $m=2$ 
and $i$ odd, $e_i$ should be set equal to $t_i(b_0, b_1, b_1, b_2)$.   

If we  had to show only
$\alpha (R \circ_{m} R ^\smallsmile ) \subseteq 
\alpha R \circ_{k} \alpha  R ^\smallsmile $,
for $\alpha$ a congruence,
it would be enough to consider the above elements, since, 
$e_0=b_0= a$ by \eqref{buh},  
$e_{k} = b_{m} = c$ by \eqref{uff}  and, 
say, for $i$ even,
$e_i  \allowbreak = \allowbreak 
t_{i+1}(b_0, b_0,b_2, b_2, \dots , \allowbreak b_{m-2},
\allowbreak  b_{m-2}, \allowbreak b_{m}, b_{m}) 
\mathrel{ R}
t_{i+1}(b_0, b_1,b_1, b_3, \dots, \allowbreak b_{m-3}, 
\allowbreak b_{m-1}, \allowbreak  b_{m-1}, b_{m})
=
e_{i+1}
$, since $R$ is admissible. 
Similarly 
$e_i \mathrel { R ^\smallsmile } e _{i+1}  $,
for $i$ odd. 
Notice that 
$b_h \mathrel R b _{h+1} $ for $h$ even
and 
$b_h \mathrel R ^\smallsmile b _{h+1} $ for $h$ odd,
hence 
$b_{h+1} \mathrel R ^\smallsmile  b _{h} $ for $h$ even
and 
$b_{h+1} \mathrel R  b _{h} $ for $h$ odd.
Moreover, for 
$i$ even, if $\alpha$ is a congruence, then
$e_i= 
t_{i}(a,a, \allowbreak  b_2, \dots, \allowbreak c )
\mathrel \alpha t_{i}(a,a, \allowbreak  b_2, \dots, a) =
a = t_{i+1}(a,b_1,  b_1, \dots, a ) \mathrel \alpha 
t_{i+1}(a,b_1,  b_1, \dots, c )  = e _{i+1} $, by
\eqref{buhh},  hence 
$e_i \mathrel \alpha  e _{i+1} $. 
Similarly, 
$e_i \mathrel \alpha  e _{i+1} $, for 
$i$ odd.
Hence the elements
$e_i$, for $0 \leq i \leq k$, witness  
$(a,c) \in \alpha R \circ_{k} \alpha  R ^\smallsmile $. 
However, as we mentioned, the identity 
$\alpha (R \circ_{m} R ^\smallsmile ) \subseteq 
\alpha R \circ_{k} \alpha  R ^\smallsmile $
is too
weak to chain back to the other conditions.

Hence we need to use the assumption
 $R= S_0 \circ S_1 \circ  \dots  \circ S_ \ell$
to prove the stronger
conclusion
$\alpha (R \circ_{m} R ^\smallsmile ) \subseteq 
\Theta \circ_ {k} \Theta ^\smallsmile $.
Moreover, we shall also 
extend some ideas from
Cz\'edli and  Horv\'ath \cite{CH} in order to 
prove the case of  (C) in which $\alpha$ 
is only assumed to be  a tolerance.

By the assumption
 $R= S_0 \circ S_1 \circ  \dots  \circ S_ \ell$
and given the elements 
$b_0, b_1, b_2, \dots, \allowbreak  b_{m}$
introduced at the beginning, we have that,
for every 
$h$ with $0 \leq h < m$,
there are elements 
$b _{h, 0},b _{h, 1}, \dots , b _{h,   \ell+1} $
such that    
$b_h = b _{h, 0} \mathrel {S_0 } b _{h, 1}
\mathrel {S_1 } \dots \mathrel {S_\ell } b _{h,   \ell+1}= b _{h+1}   $,
for $h$ even, and  
$b _{h+1}   = b _{h, 0} \mathrel {S_0 } b _{h, 1}
\mathrel {S_1 } \dots \mathrel {S_\ell } b _{h,   \ell+1}= b _{h}   $,
 for $h$ odd, since in this latter case 
$b_h \mathrel {R ^\smallsmile } b_{h+1}$,
that is,  
 $ b_{h+1} \mathrel R b_h$. 
Hence, with the $e_i$'s  defined as in the above-displayed formula,
 we have, for $i$ even,  
\begin{multline*} 
e_i =
t_{i+1}(b_0, b_0,b_2,  \dots ,  b_{m-2}, b_{m}, b_{m}) 
=
\\
t_{i+1}(b_0, b_{0,0},b_{1,0},  \dots ,  b_{m-2,0}, b_{m-1,0}, b_{m}) 
\mathrel {S_0}
\\
t_{i+1}(b_0, b_{0,1},b_{1,1},  \dots ,  b_{m-2,1}, b_{m-1,1}, b_{m})  
\mathrel {S_1}
\\
t_{i+1}(b_0, b_{0,2},b_{1,2},  \dots ,  b_{m-2,2}, b_{m-1,2}, b_{m})  
\mathrel {S_2} \dots 
\\
\mathrel {S_ \ell}
t_{i+1}(b_0, b_{0,\ell+1},b_{1,\ell+1},  \dots ,
  b_{m-2,\ell+1}, b_{m-1,\ell+1}, b_{m})  =
\\
t_{i+1}(b_0, b_1,b_1,  \dots ,  b_{m-1}, b_{m-1}, b_{m})  =
e _{i+1} 
\end{multline*}   
 Moreover, for $0 \leq q \leq \ell $ and $0 \leq i \leq k$, we have, 
by \eqref{buhh}: 
\begin{multline*} 
t_{i+1}(a, \allowbreak  b _{0, q}, \allowbreak  b _{1, q}, \allowbreak 
 \dots,c ) 
= \\
t_{i+1}(t_{i+1}(a, \allowbreak  b _{0, q}, \allowbreak  b _{1, q}, \allowbreak 
 \dots,\bm c) , b _{0, q+1}, \allowbreak  b _{1, q+1}, \allowbreak 
 \dots,t_{i+1}(\bmw a, \allowbreak  b _{0, q}, \allowbreak  b _{1, q},  \allowbreak 
 \dots,c ) )
\mathrel \alpha  \\
t_{i+1}(t_{i+1}(a, \allowbreak  b _{0, q}, \allowbreak  b _{1, q}, \allowbreak 
 \dots, \bm a) , b _{0, q+1}, \allowbreak  b _{1, q+1}, \allowbreak 
 \dots,t_{i+1}(\bmw c, \allowbreak  b _{0, q}, \allowbreak  b _{1, q},  \allowbreak 
 \dots,c ) )
=\\
t_{i+1}(a , b _{0, q+1}, \allowbreak  b _{1, q+1}, \allowbreak 
 \dots,c )
\end{multline*}
(elements in bold are those moved by $\alpha$). Hence, for $i$ even,  the elements
$ t_{i+1}(a, b _{0, q}, \allowbreak  b _{1, q},
\allowbreak  b _{2, q}, \allowbreak 
 \dots,c) $, for $0 \leq q \leq \ell $,
witness 
$e_i \mathrel { \Theta } e _{i+1}  $,
recalling that
 $\Theta = \alpha S_0 \circ \alpha S_1 \circ  \dots  \circ \alpha S_ \ell$ 
and that $a=b_0$, $b_m=c$.  
Similarly 
$e_i \mathrel { \Theta  ^\smallsmile } e _{i+1}  $,
for $i$ odd.
After the above considerations, we see that 
the elements 
$e_i$, for $0 \leq i \leq k$,
witness  
$(a,c) \in \alpha \Theta  \circ_{k} \alpha  \Theta  ^\smallsmile $,
what we had to show.

The case $m$ odd is similar. This time, the $e_i$'s are defined as follows, again for
 $0 \leq i \leq k$.
\begin{align*} \labbel{eodd}  
e_i & =  t_{i}(b_0, b_0,b_2, b_2, \dots, b_{m-3},b_{m-1}, b_{m-1}, b_{m}) 
&& \text{ for $i$ even,}
\\
e_i &=   t_{i}(b_0, b_1,b_1, b_3, \dots, b_{m-2}, b_{m-2}, b_{m}, b_{m}) 
 && \text{ for $i$ odd}.
 \end{align*} 
The rest is similar.

(C) $\Rightarrow $  (A)
Take 
$\ell=1$,
$S_0= \beta $ and 
 $S_1= \gamma  $
 in (C).
Apparently,
computing 
$R \circ_{m} R ^\smallsmile 
=
( \beta \circ \gamma ) \circ _m (\gamma \circ \beta )$
gives
$2m$ factors,
but we have 
$m-1$ adjacent pairs of  
 the same congruence, hence 
$R \circ_{m} R ^\smallsmile = \beta  \circ _{m+1} \gamma  $.
Similarly,
$\Theta \circ_ {k} \Theta ^\smallsmile 
=
\alpha \beta \circ _{k+1} \alpha \gamma $
and we get (A).
 \end{proof}

\begin{corollary} \labbel{ell}
If
$  J _{ \mathcal V}  (m) = k$ 
and $\ell> 0$, then 
$ J _{ \mathcal V} (m \ell) \leq k \ell$. 
 \end{corollary} 

\begin{proof} 
We assume $  J _{ \mathcal V}  (m) = k$ and we  have to show that 
$ \alpha ( \beta \circ _{m \ell +1 } \gamma )
 \subseteq \alpha \beta \circ_{k \ell +1} \alpha \gamma $. 
We  apply  Theorem \ref{smile}(A) $\Rightarrow $  (C),
taking $S_0=S_2= \dots = \beta $
and 
 $S_1=S_3= \dots = \gamma  $.
We have $R \circ_{m} R ^\smallsmile
= \beta \circ _{m \ell+1} \gamma $,
since, as in the proof of (B) $\Rightarrow $  (C) above, we apparently have
$ m(\ell+1)$ factors, but $m-1$ many of them 
are absorbed, hence we end up with 
 $ m(\ell+1) - (m-1) =  m \ell  +1 $
factors.
The definition of $\Theta$ in (C) becomes
$\Theta= \alpha \beta \circ _{\ell+1} \alpha \gamma  $,
hence, arguing as above,  
$\Theta \circ_k \Theta ^\smallsmile $ 
has $ k\ell   +1 $ actual
factors, that is, 
$\Theta \circ \Theta ^\smallsmile =
\alpha \beta \circ _{k\ell   +1} \alpha \gamma  $. 
The inclusion \eqref{uuu}  thus gives the corollary.
\end{proof}

Stronger results can be proved in the case of 
$3$-distributivity.
Actually, the arguments work 
in a  more general setting.
In the following theorem we do not need the assumption that 
$\mathcal V$ is congruence distributive.

\begin{theorem} \labbel{4gt}
Suppose that $\mathcal V$ satisfies the congruence identity
\begin{equation}\labbel{4}     
\alpha( \beta \circ \gamma ) \subseteq
 \alpha (\gamma  \circ \beta  ) \circ \alpha \gamma  
  \end{equation}
(the above identity 
is equivalent  to the existence of $3$ Gumm terms
$p, j_1, j_2$ in the terminology of Section \ref{msp} below).
Then $\mathcal V$ satisfies
\begin{gather} \labbel{4a}
\alpha ( \beta  \circ \gamma  ) \circ \alpha \beta 
=
\alpha \beta  \circ \alpha ( \gamma \circ \beta ) 
\\
\labbel{4b}
 \alpha( \beta \circ \gamma \circ \beta ) \circ \alpha \gamma    =
  \alpha ( \beta \circ \gamma )    \circ \alpha \beta  \circ \alpha \gamma  
\\
\labbel{4c}
 \alpha( \beta \circ _{m+2}  \gamma  )   =
  \alpha ( \beta \circ \gamma )    \circ (\alpha \beta  \circ_m \alpha \gamma), \quad
\text{ for } m \geq 2.   
 \end{gather}   
 \end{theorem}

Notice that the identity
\eqref{4} is weaker than  $3$-distributivity, which
corresponds to the identity    
 $\alpha( \gamma   \circ \beta ) \subseteq 
\alpha \gamma \circ \alpha  \beta \circ \alpha \gamma  $, equivalently,
 taking converse,
$\alpha( \beta \circ \gamma ) \subseteq
 \alpha \gamma  \circ \alpha \beta   \circ \alpha \gamma$.

\begin{proof}
By applying  \eqref{4}
with $\beta$ and $\gamma$ exchanged, we get
$\alpha \beta  \circ \alpha ( \gamma \circ \beta ) \subseteq
\alpha \beta  \circ \alpha ( \beta \circ \gamma  ) \circ \alpha \beta =
\alpha ( \beta \circ \gamma  ) \circ \alpha \beta $,
since obviously
 $\alpha \beta  \circ \alpha ( \beta \circ \gamma  )=
 \alpha ( \beta \circ \gamma  )$. 
The reverse inclusion in \eqref{4a}  follows by symmetry.

It is easy and standard to show that
if $\mathcal V$ satisfies the congruence identity 
\eqref{4} then  
 $\mathcal V$ 
has terms $p$ and  $j$ such that the identities 
$x=p(x,y,y)$, $p(x,x,y)=j(x,x,y)$, $j(x,y,y)=y$
and $x=j(x,y,x)$ 
hold in $\mathcal V$.
Were $\mathcal V$  in addition
$3$-distributive, we 
would also have the identity  $x=p(x,y,x)$.  
Cf. the J{\'o}nsson terms which shall be recalled
at the beginning of the next section. 
However, here the identity 
 $x=p(x,y,x)$ shall not be needed.

We now  prove \eqref{4b}.
One inclusion is trivial and, in order to prove the nontrivial 
inclusion, it is enough to prove
$\alpha( \beta \circ \gamma \circ \beta )  \subseteq 
\alpha ( \beta \circ \gamma ) \circ \alpha \beta \circ \alpha \gamma    $,
since $\alpha \gamma $ is a congruence, hence transitive.
So let $(a,d) \in \alpha( \beta \circ \gamma \circ \beta ) $,
hence $a \mathrel \alpha  d$ and $a \mathrel { \beta } b \mathrel { \gamma  }
c \mathrel { \beta } d   $, for some $b$ and  $c$.
Let us compute  
$a=j(d, a, a) \mathrel {\beta  } j(d, b, a)
\mathrel { \gamma } j(d, c, a)
\mathrel \beta 
j(c, c, b) = p(c, c, b) \mathrel { \beta  }  p(d, c, b)
\mathrel {\gamma  } p(d,c, c)= d$.
Moreover,
$a = j(a, b, a) 
\mathrel \alpha  
j(d, b, a)$,
hence 
$ a \mathrel {  \alpha  \beta } j(d,b,a) $. 
Furthermore, 
$j(d, b, a) 
\mathrel \alpha  a=
j(a, c, a) \mathrel \alpha  j(d,c,a) $
hence also
$j(d, b, a) 
\mathrel { \alpha \gamma }   j(d,c,a) $.
Finally,
$j(d,c,a) \mathrel \alpha  a \mathrel \alpha  d$,
hence 
 $j(d,c,a) \mathrel { \alpha ( \beta \circ \gamma )}   d$.
Hence the elements 
$ j(d, b, a) $ and $ j(d, c, a)$
witness 
$(a,d) \in 
 \alpha \beta  \circ \alpha \gamma \circ \alpha (  \beta \circ \gamma  )$.  
By applying \eqref{4a}  twice, we get
$ \alpha \beta  \circ \alpha \gamma \circ \alpha (  \beta \circ \gamma  ) =
 \alpha \beta  \circ \alpha (\gamma \circ   \beta) \circ \alpha \gamma  =
 \alpha (  \beta \circ \gamma  ) \circ \alpha \beta  \circ \alpha \gamma $. 
In conclusion, 
$(a,d) \in  \alpha (  \beta \circ \gamma  ) \circ \alpha \beta  \circ \alpha \gamma $
and \eqref{4b} is proved. 

In order to prove \eqref{4c} we need a claim.

\begin{claim}  
If
$\mathbf A \in \mathcal V$, $\alpha$ is a congruence,
$R$, $S$, $T$  are reflexive and admissible relations on $\mathbf A$, 
$R \subseteq T$ and $S \subseteq T ^\smallsmile $, then
\begin{equation}\labbel{ru}  
\alpha (T \circ T ^\smallsmile  \circ R \circ S ) =
\alpha (T \circ T ^\smallsmile ) \circ \alpha R \circ \alpha S 
 \end{equation}     
 \end{claim}

To prove the claim, first notice that an inclusion is trivial,
since $\alpha$ is assumed to be a congruence.
To prove the nontrivial inclusion, 
let $(a,e) \in \alpha(T \circ T ^\smallsmile  \circ R \circ S ) $,
with  $a \mathrel \alpha  e$ and $a \mathrel { T} b \mathrel {T ^\smallsmile }
c \mathrel { R } d  \mathrel {S  } e  $.
We have
$a=p(a, c, c) \mathrel { T } p(b, b, d) =j(b,b,d)
\mathrel { T ^\smallsmile } j(a, c, e) $,
since 
$R \subseteq T$ and $S \subseteq T ^\smallsmile $.
Moreover, 
$a=j(a,c,a) \mathrel \alpha  j(a,c,e)$,
hence
 $a \mathrel {\alpha (T \circ T ^\smallsmile )}  j(a,c,e)$.
Furthermore,
$ j(a,c,e) \mathrel R j(a,d,e) \mathrel { S}  j(a,e,e) =e$
and 
 $ j(a,c,e)  \mathrel \alpha a \mathrel \alpha  e =  j(e,d,e) \mathrel {  \alpha }  j(a,d,e) $, hence
$ j(a,c,e) \mathrel { \alpha R} j(a,d,e)$
and $ j(a,d,e) \mathrel { \alpha  S}  =e$,
thus the elements 
$ j(a, c, e) $ and $ j(a, d, e)$
witness 
$(a,e) \in \alpha (T \circ T ^\smallsmile ) \circ \alpha R \circ \alpha S$
and the claim is proved.  

Having proved the claim,
we go on by proving the case $m=2$ of equation \eqref{4c}. 
By taking
 $T =  \beta \circ \gamma  $,
$R= 0 $ and
 $S= \gamma $
 in equation \eqref{ru},
we get 
$ \alpha( \beta \circ \gamma \circ \beta \circ \gamma  )=
  \alpha( \beta \circ \gamma \circ \gamma \circ \beta  \circ \gamma  )=
 \alpha( T \circ T ^\smallsmile  \circ R \circ S  )
= 
\alpha (T \circ T ^\smallsmile  ) \circ \alpha R \circ \alpha S=
\alpha ( \beta \circ \gamma \circ \gamma \circ \beta   )
 \circ \alpha  \gamma =
\alpha ( \beta \circ \gamma \circ \beta   )
\circ \alpha  \gamma =
\alpha (\beta \circ   \gamma ) \circ \alpha \beta \circ \alpha \gamma  $, 
where in the last identity  we have used  \eqref{4b}.

The rest of the proof now follows 
quite easily by induction.
Suppose that $m \geq 3$ and that  we have proved 
equation \eqref{4c} for every $n$ with  
 $2 \leq n < m $. 
If $m$ is odd, take 
$T = \beta  \circ _{ \frac{m+1}{2}}  \gamma  $,
 $R= \gamma $ and 
 $S= \beta  $ in \eqref{ru}, 
getting
$ \alpha ( \beta \circ _{m+2} \gamma  ) =
\alpha (T \circ T ^\smallsmile \circ R \circ S  ) =
 \alpha (T \circ T ^\smallsmile ) \circ \alpha R \circ \alpha S=
 \alpha ( \beta \circ _{m} \gamma  ) \circ \alpha \gamma \circ \alpha \beta =
 \alpha (\beta  \circ \gamma ) \circ ( \alpha \beta \circ  _{m} \alpha \gamma ) $,
where the  last identity follows from the inductive hypothesis,
except in case $m=3$, where we need \eqref{4b}. 
If $m$ is even, take 
$T = \beta  \circ _{ \frac{m+2}{2}}  \gamma  $,
 $R= 0 $ and 
 $S= \gamma  $ in \eqref{ru}, 
getting
$ \alpha ( \beta \circ _{m+2} \gamma  ) =
\alpha (T \circ T ^\smallsmile \circ R \circ S  ) =
 \alpha (T \circ T ^\smallsmile ) \circ \alpha R \circ \alpha S=
 \alpha ( \beta \circ _{m+1} \gamma  ) \circ \alpha \gamma =
 \alpha (\beta  \circ \gamma ) \circ (\alpha \beta \circ  _{m} \alpha \gamma ) $,
where we have used the inductive hypothesis again to obtain the last identity.
 \end{proof}    

\begin{corollary} \labbel{th3d}
If $ J _{ \mathcal V} (1) = 2 $,
that is, $\mathcal V$ is $3$-distributive, 
then
$ J _{ \mathcal V} (n) \leq  n$,
for every $n \geq 3$.
Moreover the following congruence identity holds in $\mathcal V$ 
\begin{equation*}
     \alpha( \beta \circ \gamma \circ \beta ) \circ \alpha \gamma    =
\alpha \beta \circ \alpha \gamma \circ  \alpha \beta \circ \alpha \gamma    
  \end{equation*}
 \end{corollary} 

\begin{proof}
Immediate from Theorem \ref{4gt}.
Indeed, as we mentioned, if $\mathcal V$ is $3$-distributive, then
the hypothesis of Theorem \ref{4gt} holds. 
Then,  using 
$ \alpha ( \beta \circ \gamma ) 
\subseteq \alpha \beta \circ \alpha \gamma \circ \alpha \beta $ in
equation \eqref{4b},
we get
 $     \alpha( \beta \circ \gamma \circ \beta ) \circ \alpha \gamma   
 =
     \alpha( \beta \circ \gamma) \circ \alpha \beta  \circ \alpha \gamma 
\subseteq  
\alpha \beta \circ  \alpha \gamma \circ \alpha \beta \circ  \alpha \beta \circ \alpha \gamma =
\alpha \beta \circ  \alpha \gamma \circ  \alpha \beta \circ \alpha \gamma $. 
The converse inclusion is trivial.
All the other inclusions are similar; as above, in any case,
a pair of $\alpha \beta $'s is absorbed into a single occurrence
(the case $n=3$ here corresponds to $m=2$ in \eqref{4c} and so on). 
 \end{proof}

\section{Variants of the spectrum}\labbel{var}

One can consider an alternative function
$ J _{ \mathcal V} ^\smallsmile  $
which is defined in such a way that 
$ J _{ \mathcal V} ^\smallsmile (m) $
is the smallest $k$ such that 
the identity  
\begin{equation}\labbel{d3b}  
\tag*{$(m+1,k+1)$-dist$ ^\smallsmile  $ }
 \alpha ( \beta \circ _{m+1} \gamma ) \subseteq 
\alpha \gamma  \circ _{k+1} \alpha \beta 
\end{equation}      
holds in $\mathcal V$.
Here $\gamma$ and $\beta$ are exchanged on the
right-hand side, in comparison with $(m+1,k+1)$-dist.
It is easy to see that $ J _{ \mathcal V} ^\smallsmile $
and $ J _{ \mathcal V} $ are different functions;
indeed,  $ J _{ \mathcal V} ^\smallsmile (m) = m$,
for some $m$, 
implies $m+1$-permutabilty, while 
$ J _{ \mathcal V} (m) = m$ holds in lattices,
for every $m$.
Obviously, however, 
$ J _{ \mathcal V} ^\smallsmile (m) $
and
$ J _{ \mathcal V}  (m) $
differ at most by $1$.
There are some further simple relations connecting 
$ J _{ \mathcal V} ^\smallsmile  $
and
$ J _{ \mathcal V}   $.
For example, if  $m$ and
$ J _{ \mathcal V}  (m) $ have the same parity, 
then $ J _{ \mathcal V} ^\smallsmile (m) \geq J _{ \mathcal V}  (m) $.
Compare the parallel discussion
(corresponding to  the case $m=1$ here) in \cite[p. 63]{FV}.  

A theorem analogous to Theorem \ref{smile}
holds for $ J _{ \mathcal V} ^\smallsmile  $:
just exchange even and odd in \eqref{aa} and replace 
identity \eqref{uuu} by   
$\alpha (R \circ_{m} R ^\smallsmile ) \subseteq 
\Theta ^\smallsmile \circ_ {k} \Theta $. 
Arguing as in Corollary \ref{ell},
we then get that 
if
$  J _{ \mathcal V} ^\smallsmile  (m) = k$ 
and $\ell$ is odd, then 
$ J _{ \mathcal V} ^\smallsmile (m \ell) \leq k \ell$. 
If $\ell$ is even, then
 $  J _{ \mathcal V} ^\smallsmile  (m) = k$  implies
$ J _{ \mathcal V}  (m \ell) \leq k \ell$.

A probably more significant variant is suggested by
the use of directed J{\'o}nsson terms.
See Z\'adori \cite{Z} and  Kazda,  Kozik,  McKenzie and Moore
\cite{kkmm}. Let us recall the definitions.
\emph{J{\'o}nsson terms} \cite{JD}  are terms 
$j_0, \dots, j_k$ satisfying 
\begin{align} \labbel{j1} 
\tag{J1}
  x&= j_0(x,y,z),  &  \\
\labbel{j2}
\tag{J2}
  x&=j_i(x,y, x), 
\quad \ \ \  \text{ for } 
0 \leq i \leq k,
 \\ 
 \labbel{j3} \tag{J3} 
\begin{split}  
 j_{i}(x,x,z) &=
j_{i+1}(x,x,z),
 \quad \text{ for even $i$,\ } 
0 \leq i < k,
  \\ 
 j_{i}(x,z,z)&=
j_{i+1}(x,z,z),
  \quad \text{ for odd $i$,\ } 
0 \leq i < k,
 \end{split}
 \\ 
\labbel{j4}
\tag{J4}
j_{k}(x,y,z)&=z  &
\end{align}   
Notice that 
this is exactly condition (B) in Theorem \ref{smile} 
in the particular case  $m=1$ and with 
$k$ in place of $k+1$.  
We get 
  \emph{directed J{\'o}nsson terms},
or \emph{Z\'adori  terms}  
\cite{Z,kkmm} 
if
in the above set of identities we replace condition
\eqref{j3} by
\begin{equation}\labbel{D}
\tag{D}
  j_{i}(x,z,z)=
j_{i+1}(x,x,z)  \quad \text{ for \ } 
0 \leq i < k
  \end{equation}    

Seemingly, directed J{\'o}nsson terms first appeared (unnamed)
in \cite[Theorem 4.1]{Z},
whose proof  relies on McKenzie \cite{Mc}.
In \cite{Z} 
 it is shown, among other,  that  a finite  bounded poset 
admits J{\'o}nsson operations for some
$k$ if and only if it admits directed J{\'o}nsson operations for some
$k'$.
Kazda,  Kozik,  McKenzie and Moore
\cite{kkmm} proved the equivalence for
terms in an arbitrary variety,
 thus, by \cite{JD}, a variety is congruence distributive if and only if
it has directed J{\'o}nsson terms.

For a binary relation $R$,
let $R^k = R \circ R \circ R \dots$
with $k$ factors. In other words,
 $R^k = R \circ_ k R $.

\begin{proposition} \labbel{kk}
If some variety $\mathcal V$ has $k+1$ 
directed J{\'o}nsson terms
 $d_0, \dots, \allowbreak  d _{k} $, with $k \geq 1$, 
then, for every $\ell \geq 1$, $\mathcal V$ satisfies the identity
\begin{equation*}\labbel{ars}
\alpha (S_1 \circ S_2 \circ \dotsc \circ S_ \ell  ) \subseteq 
(\alpha S_1 \circ \alpha S_2 \circ \dotsc \circ \alpha S_ \ell  )^{k-1}
  \end{equation*}    
where $\alpha$ varies among tolerances
and $S_1, S_2, \dots$ 
vary among reflexive and admissible relations
on some algebra in $\mathcal V$. In particular, for $\ell$ even,
 $\mathcal V$ satisfies
\begin{equation}\labbel{arss}
\alpha ( S \circ _{ \ell} T ) \subseteq
\alpha  S \circ _{ \ell (k-1)} \alpha T 
  \end{equation}    
and, for $\ell$ odd, 
\begin{equation*}\labbel{arsss}
\alpha ( \beta  \circ _{ \ell} T ) \subseteq
\alpha  \beta  \circ _{ k'} \alpha T, 
  \end{equation*}
where $\beta$ varies among congruences and $k'=\ell (k-1) -k +2$.     
\end{proposition}

 \begin{proof} 
Let us work in some fixed algebra belonging to $\mathcal V$.
Let $(a,c) \in \alpha (S_1 \circ S_2 \circ \dots \circ S_ \ell  )$.
Thus $ a \mathrel \alpha  c$
and there are elements
$b_0, b_1, b_2, \dots, b_{\ell} $
such that  $a=b_0 \mathrel {S_1}  b_1 \mathrel {S_2}  b_2 \dots
b_{\ell-1} \mathrel {S_\ell} b_{\ell}= c$.
First suppose that $\alpha$ is a congruence.
For every $i$ and $h$,
we have  
$d_i(a,b_h,c) \mathrel \alpha d_i(a,b_h,a)= a  $,
hence all such elements are $\alpha$-related.
Moreover, for every $i < k$, 
$d_i(a,a,c) \mathrel {S_1}  d_i(a,b_1,c)
\mathrel {S_2}  d_i(a,b_2,c) \mathrel {S_3}  \dots
\mathrel {S_\ell}  d_i(a,c,c) =  d_{i+1}(a,a,c)  $.
This shows that, for every $i < k$,
$(d_i(a,a,c) ,  d_{i+1}(a,a,c) ) 
 \in \alpha S_1 \circ \alpha S_2 \circ \dots \circ \alpha S_ \ell $.

Since $a=d_0(a,c,c) = d_1(a,a,c)$
and $d _{k} (a,a,c) = c $,
then the elements   
$d _{i} (a,a,c)  $, for
$1 \leq i \leq k $, 
witness 
$(a,c) \in (\alpha S_1 \circ \alpha S_2 \circ \dots \circ \alpha S_ \ell  )^{k-1}$.

The last statement follows immediately.
We just mention that in the last equation
$k-2$ factors are absorbed, since we assume that $\beta$  
is transitive.

The case when $\alpha$ is just  a tolerance is treated as in
Cz\'edli and  Horv\'ath \cite{CH}.
Indeed, for every 
 $i,j,h,h'$,
we have  
$d_i(a,b_h,c) = \allowbreak d_i(d_j(a, \allowbreak b_{h'},a), \allowbreak b_h,d_j(c,b_{h'},c))
\allowbreak \mathrel \alpha 
d_i(d_j(a,b_{h'},c), \allowbreak b_h,d_j(a,b_{h'},c)) \allowbreak =
d_j(a,b_{h'},c) $. The rest is the same.
\end{proof}    

Since congruences and tolerances are, in particular, reflexive and admissible,
we  obtain corresponding congruence/tolerance  
identities from the above identities about relations. 
In particular, by \cite{JD,kkmm}
and Proposition \ref{kk}, 
a variety is congruence distributive if and only if 
equation \eqref{arss} holds for some 
$\ell \geq 2$ and $k$, equivalently,  
for $\ell=2$ and some $k$.
If $^*$ denotes  transitive closure,
we then get that a variety $\mathcal V$ is congruence distributive
if and only if  
$(\alpha ( S \circ  T ))^* =
(\alpha  S \circ  \alpha T )^*$
holds in $\mathcal V$, 
if and only if 
$\alpha^* ( S \circ  T )^* =
(\alpha  S \circ  \alpha T )^*$
holds in $\mathcal V$.
Notice that, on the other hand,
neither the identity 
  $(\alpha (S \circ S) )^* =
(\alpha  S  )^*$
nor
the identity
  $\alpha ^*S  ^* =
(\alpha  S  )^*$
 imply congruence distributivity, since the identities hold, 
e.~g., in permutable varieties
(in fact, we have a proof that both identities are equivalent
to congruence modularity \cite{ricm}).
See however \cite{uar} for a variation actually equivalent 
to congruence distributivity. 
 In all the above statements
$\alpha$ can be taken equivalently as a tolerance or as a congruence,
while $S$ and $T$ vary among reflexive and admissible relations.

We do not know
whether, for every
congruence distributive variety $\mathcal V$, there is 
$k$ such that $\mathcal V$ satisfies the relation identity
$R (S \circ T) \subseteq RS \circ_k RT$.
The existence of some $k$ as above  is equivalent to 
$R^*( S \circ  T )^* =
(R S \circ  RT )^*$.
 The arguments from
Gyenizse and  Mar\'oti 
\cite{GM}
can be adapted to show that if
$\mathcal V$ is congruence distributive and 
the free algebra $\mathbf F_{ \mathcal V}(3)$
 in $\mathcal V$ over $3$ elements
is finite, then 
$R^*( S \circ  T )^* =
(R ^*S \circ  R^*T )^*$.
It is easy to see  that $2$-distributivity does imply 
$R (S \circ T) \subseteq RS \circ RT$.
Cf.~\cite[Remark 17]{contol}.
We have a proof that if $\mathcal V$ has $4$ 
directed J{\'o}nsson terms
 $d_0, d_1, d_2, d _{3} $
and $\mathbf F_{ \mathcal V}(2)$
is finite, then 
$R( S \circ  T ) \subseteq 
R S \circ_k  RT $,
for some $k$ which depends on the variety.

Notice also that, since the composition of reflexive and admissible relations 
is still reflexive and admissible, we can get new identities by substitution, without
recurring to terms. 
E.~g., from $ \alpha (S \circ T) \subseteq \alpha S \circ \alpha T \circ \alpha S$,
replacing  $T$ by $T \circ S$, we get 
$ \alpha (S \circ T \circ S) \subseteq 
\alpha S \circ \alpha (T \circ S) \circ \alpha S
\subseteq
\alpha S \circ \alpha T \circ  \alpha S \circ \alpha T \circ \alpha S
$.
We leave similar computations to the interested reader,
observing that  arguments using terms, though
more involved, generally produce stronger results.

By the above considerations, one would be tempted to believe that 
it is always more convenient to use directed J{\'o}nsson terms,
rather than
 J{\'o}nsson terms. However,
the exact relation  between 
the (minimal) number of directed J{\'o}nsson terms 
and of J{\'o}nsson terms in a variety is not clear;
see \cite[Observation 1.2 and Section 7]{kkmm}.
Even in the case when we have the same number of 
J{\'o}nsson and of directed J{\'o}nsson terms,
there are situations in which it is more convenient to use the former
terms. For example, with $5$ J{\'o}nsson terms 
we have
$\alpha( \beta \circ \gamma  ) \subseteq \alpha \beta \circ _4 \alpha \gamma $,
while from  
$5$ J{\'o}nsson directed terms 
we obtain only
$\alpha( \beta \circ \gamma ) \subseteq \alpha \beta \circ _6 \alpha \gamma $
from 
\eqref{arss}.
Notice that here we  count terms 
including the initial and final projections
(that is,  $5$ terms correspond to the case $k=4$). 

On the other hand, an example in which it is more convenient to use 
directed terms  is Baker example from \cite{B}, 
the variety $\mathcal V$ generated by the
polynomial  reducts of lattices in which 
only the ternary operation $ f(a,b,c) =a \wedge (b \vee c)$ 
is considered. See also \cite[Example 2.12]{cd}. 
Baker showed that $\mathcal V$ is $\Delta_4$
but not $\Delta_3$, that is, in the present terminology,
$ J _{ \mathcal V} (1) =3 $, from which we get from Corollary \ref{ell} that 
$ J _{ \mathcal V} ( \ell ) \leq 3 \ell  $.
However, taking  $d_1(x,y,z)=x \wedge (y \vee z)$, 
$d_2(x,y,z)=z \wedge (x \vee y) $ and $d_0$ and  $d_3$
the first and last projections, we get a sequence of 
directed J{\'o}nsson terms, hence we can apply  
Proposition \ref{kk}  with $k=3$ in order to get
the better evaluation 
$ J _{ \mathcal V} ( \ell -1 ) \leq 2\ell -2 $,
for $\ell$ odd and
$ J _{ \mathcal V} ( \ell -1 ) \leq 2\ell  -1 $,
for $\ell$ even.

We can introduce a version of $ J _{ \mathcal V} $ 
which takes into account reflexive and admissible relations
in place of congruences. For $m \geq 1$, we 
let $ J^r _{ \mathcal V} (m) $ be the least 
$k$ such that   
$\alpha ( S \circ _{ m+1} T ) \subseteq
\alpha  S \circ _{ k+1} \alpha T $;
similarly,  
$ J _{ \mathcal V}^{r \smallsmile }$
is defined by considering 
$\alpha  T \circ _{ k+1} \alpha S $ 
on the right-hand side.
The definitions are justified in view of 
\cite{JD,kkmm} and 
Proposition \ref{kk}.
We can thus formulate the following general problem.

\begin{JDSPG} 

\ 

\noindent 
\emph{\ 
Which quadruplets of functions  can be realized as
 $(J _{ \mathcal V}, J _{ \mathcal V}^{\smallsmile },
\allowbreak J _{ \mathcal V}^{r}, J _{ \mathcal V}^{r \smallsmile }) $,
for some congruence distributive variety  $\mathcal V$?}
\end{JDSPG}

In the above problem 
we could take into account additional conditions. For example,
consider the formula
$\alpha( \beta \circ \gamma \circ \delta \circ \beta \dots )
\subseteq \alpha\beta \circ \alpha \gamma \circ \alpha \delta \circ \alpha \beta 
\dots$,
 with $m$ 
factors on the left-hand side.
Which is the smallest number of factors
on the right-hand side such that
the formula holds in a specific congruence
distributive variety? 
Similarly, which is the best bound in
$(\alpha_1 \circ_m \beta _1)(\alpha_2 \circ_m \beta _2)
\subseteq  \alpha _1 \alpha _2 \circ \alpha _1 \beta _2 \circ  
\beta  _1 \alpha _2 \circ \beta  _1 \beta _2  \circ  \alpha _1 \alpha _2  \dots$?
Or, more generally, for
$(\alpha_1 \circ_m \beta _1)(\alpha_2 \circ_m \beta _2)
\dots (\alpha_h \circ_m \beta _h)$? 
What about 
$(\alpha \circ_m \beta )(\alpha \circ_m \gamma )( \beta  \circ_m \gamma )
\subseteq
\alpha \beta \circ \beta \gamma \circ \alpha \gamma \circ \alpha \beta  \dots$? 

Of course, the analogue of Proposition \ref{nip}
holds for $J _{ \mathcal V}^{\smallsmile }$,
$ J _{ \mathcal V}^{r}$ and $ J _{ \mathcal V}^{r \smallsmile }$,
too, using similar arguments.

\section{An unexpected  connection with 
the modularity spectra} \labbel{msp}

One can also define functions analogous to 
$ J _{ \mathcal V} $ in the case of congruence modular varieties.
In the present section we  give up  the convention that every 
variety at hand is congruence distributive.
By a fundamental theorem by A. Day \cite{D}, a variety
$\mathcal V$  is congruence modular
if and only there is some $k$ such that  the congruence identity
\begin{equation}\labbel{d}    
\tag{D$_k$}
\alpha ( \beta \circ \alpha \gamma \circ \beta )
\subseteq \alpha \beta \circ _k \alpha \gamma 
  \end{equation}
 holds in $\mathcal V$. 
Again, Day result
is  stated 
in a form involving a certain number of terms, but we 
shall not need the explicit Day terms here.
A variety $\mathcal V$ is \emph{$k$-modular}
if $\mathcal V$ satisfies equation \eqref{d}, and 
$\mathcal V$ has \emph{Day level} $k$ if such a $k$ is minimal.   
For a congruence modular variety $\mathcal V$, we can define the
 \emph{Day modularity function} $D _ { \mathcal V}$ as follows.
For $m \geq 3$,    $D _ { \mathcal V}(m)$ is the least $k$ 
such that 
$ \alpha ( \beta \circ_m \alpha \gamma )
\subseteq \alpha \beta \circ _k \alpha \gamma 
$ holds in $\mathcal V$.
The arguments from \cite{D}  
show that  $D _ { \mathcal V}(m)$ is defined for every 
$m$ and every congruence modular variety $\mathcal V$,
but the methods from \cite{D} do not furnish the best value.  
See \cite{cm}.

The case of congruence modularity is substantially
more involved than the distributivity case treated here. 
 Gumm \cite{G1,G2} 
 provided another characterization of congruence modular
varieties
by considering terms which ``compose permutability with distributivity''.
In detail, \emph{Gumm terms} are terms $p, j_1, \dots, j_k$   
satisfying 
the above conditions \eqref{j2}-\eqref{j4}
(the distributivity part, involving only 
the $j$'s), together with the following permutability part:
\begin{equation}\labbel{p} 
\tag{P}
x= p(x,z,z), \quad \text{ and } \quad p(x,x,z)= j_1(x,x,z) 
 \end{equation}

Notice that the definition given here
 is slightly different from 
Gumm original  conditions in \cite[Theorem 1.4]{G1},
where odd and even are exchanged and
where $p$ is considered after the 
$j_i$'s (the $q_i$'s in the notation from \cite{G1}).  This is not simply a matter of symmetry:
in the formulation by Gumm,
when $n$ is odd,
 one gets an unnecessary term
which can be discarded, hence Gumm actual  condition is interesting
only for $n$ even.   Of course,
the above remark is interesting only when one is
concerned with  the smallest $k$ (or $n$) for which there are Gumm
terms. If one is concerned just with the existence of some
$k$ for which such terms exist, then the remark is irrelevant.
As far as we know, the above formulation 
 first appeared in  Lakser, Taylor and Tschantz
\cite{LTT} and Tschantz \cite{T}.

Using Gumm terms, 
 Tschantz \cite{T} showed that a variety
$\mathcal V$ 
 is congruence modular
if and only if the congruence identity 
$ \alpha ( \beta + \gamma ) \subseteq  
\alpha (\gamma \circ \beta )
\circ  ( \alpha \gamma + \alpha \beta )$ 
holds in $\mathcal V$.
Hence it is also natural to introduce the
 \emph{Tschantz modularity function} $T_ { \mathcal V}$ 
for a congruence modular variety $\mathcal V$ 
in such a way that, for
 $m \geq 2$,    $T _ { \mathcal V}(m)$ is the least $k$ 
such that the following congruence identity holds in $\mathcal V$
\begin{equation*}\labbel{Teq}
 \alpha ( \beta \circ_{m} \gamma )
\subseteq 
\alpha (\gamma \circ \beta ) \circ (\alpha \gamma  \circ _{k} \alpha \beta ) 
  \end{equation*}    

Notice that 
$T _ { \mathcal V}(2)=0$
is equivalent to congruence permutability,
just take $\alpha=1 $  
(by convention we let $ \beta  \circ_0 \gamma  = 0$). More
generally, $T _ { \mathcal V}(2)=k$
if and only if $\mathcal V$ has 
$k+2$ Gumm terms  
$p, j_1, \dots, j_{k+1}$,
but   not 
$k+1$ Gumm terms. 
In the above terminology, equation \eqref{4c}
in Theorem \ref{4gt} states that
if   $T _ { \mathcal V}(2)=1$,
then  $T _ { \mathcal V}(m+2) \leq m$,
for $m \geq 2$.

The relationships between
 $D _ { \mathcal V}$
 and $T _ { \mathcal V}$
appear rather involved. 
A detailed study of their connection
goes beyond 
the scope of the present note; we refer
to \cite{cm} for further details.
We just mention that we can define the $ ^\smallsmile $-variants
of the above notions.
Moreover, Kazda,  Kozik,  McKenzie and Moore
\cite{kkmm} 
 introduced the directed variants of Gumm terms, too.
Using  their result we can see that a variety is congruence 
modular if and only if, for every $m \geq 1$,
there is some $h$ such that 
$ \alpha ( S \circ_{m+1} T )
\subseteq 
\alpha (T \circ S ) \circ (\alpha T \circ _{h} \alpha S) 
$, 
where, as above, $S$ and $T$
 vary among reflexive and admissible relations.    
See \cite{ricm}.

We finally note  an 
intriguing connection
among the  J{\'o}nsson and the modularity spectra.
We first need a preliminary lemma.

\begin{lemma} \labbel{gt}
If a variety $\mathcal V$ has $k+2$ Gumm terms $p, j_1, \dots, j_{k+1}$,
then $\mathcal V$ satisfies the following congruence identity
(as above, $\alpha$ can be taken as a tolerance).
\begin{equation*}\labbel{t3} 
\alpha ( \beta \circ \gamma \circ \beta ) 
\subseteq 
\alpha (\gamma \circ \beta ) \circ (\alpha \gamma  \circ _{2k} \alpha \beta ) 
 \end{equation*}     
 \end{lemma} 

 \begin{proof} 
If we do not care about the exact number of factors
on the right-hand side, the proposition follows from \cite{T}. 
The arguments from \cite{T} seem
 to produce a much larger number of factors, anyway. 

Suppose that
$(a,d) \in \alpha ( \beta \circ \gamma \circ \beta ) $,
thus 
$a \mathrel \alpha  d$ and 
$a \mathrel \beta b \mathrel \gamma c \mathrel \beta d$,
for some $b$ and $c$.
Using equations \eqref{j2}-\eqref{j4},
 no essentially new argument is needed to show that
$(j_1(a,a,d),d) 
\in (\alpha \beta \circ \alpha \gamma \circ \alpha \beta) ^{k} $.

Indeed, by \eqref{j2}, 
$j_i(a,e,d)  =
j_i(j_{i'}(a,e',\bm a),e,j_{i'}(\bmw d,e',d))
\mathrel \alpha  
j_i(j_{i'}(a,e', \allowbreak \bm d),e, \allowbreak j_{i'}(\bmw a,e',d))
= j_{i'}(a,e',d)$,
for all indices  $i, i'$ and elements $e, e'$, hence all the elements
of the above form
are $\alpha$-related.    
For $i$ odd, we have 
$j_i(a,a,d) \mathrel { \beta }  j_i(a,b,d)\mathrel { \gamma }  j_i(a,c,d) 
\mathrel { \beta }  j_i(a,d,d)   $,
thus $(j_i(a,a,d), \allowbreak j_i(a,d,d)) \in \alpha \beta \circ \alpha \gamma \circ \alpha \beta $.
Similarly,
$(j_i(a,d,d),j_i(a,a,d)) \in \alpha \beta \circ \alpha \gamma \circ \alpha \beta $,
for $i$ even.
Hence 
$j_1(a,a,d)$,
$j_1(a,d,d) = j_2(a,d,d)$,
$j_2(a,a,d) = j_3(a,a,d)$, \dots \ 
witness    
$(j_1(a,a,d),d) \in 
(\alpha \beta \circ \alpha \gamma \circ \alpha \beta) ^{k} =
 \alpha \beta \circ_{2k+1} \alpha \gamma $.

Moreover, 
$a= p(a, b, \bm b) \mathrel \gamma 
p(a, \bm b, \bm c) \mathrel  \beta
 p(a, \bm a,  \bm d )
= j_1(a,a,d)$, by \eqref{p}, 
hence
$(a,j_1(a,a,d)) \in \alpha ( \gamma \circ \beta )$, 
thus
$(a,d) \in \alpha (\gamma \circ \beta ) \circ
 (\alpha \beta \circ_{2k+1} \alpha \gamma )$. 
Finally,
$\alpha (\gamma \circ \beta ) \circ
 \alpha \beta = \alpha (\gamma \circ \beta )$
(or, better, use $j_1(a,b,d)$ in place of
$j_1(a,a,d)$),
hence one more factor is absorbed and
we get the conclusion. 
\end{proof}

\begin{theorem} \labbel{jgt}
Suppose that  $V$ is a congruence distributive variety.
  \begin{enumerate}    
\item 
If  $m \geq 2$ and $\mathcal V$ is $m$-modular, then
$  J _{ \mathcal V} (2) 
\leq  J ^\smallsmile  _{ \mathcal V} (1) + 2m^2-2m -2$
and  $  J ^\smallsmile  _{ \mathcal V} (2)
 \leq  J  _{ \mathcal V} (1) + 2m^2-2m -2 $.
\item
If $\mathcal V$ has $k+2$ Gumm terms, then
$  J _{ \mathcal V} (2) 
\leq  J ^\smallsmile  _{ \mathcal V} (1) + 2k  $
and  $  J ^\smallsmile  _{ \mathcal V} (2)
 \leq  J   _{ \mathcal V} (1) + 2k $.
The bounds can be improved by $1$ if 
  $J ^\smallsmile  _{ \mathcal V} (1)$ is odd, 
respectively, if $J   _{ \mathcal V} (1)$ is even. 
  \end{enumerate} 
 \end{theorem}

\begin{proof} 
Part (2) is immediate from Lemma \ref{gt}.

In \cite[Theorem 2]{LTT} it is proved that if a variety 
is $m$-modular, then $\mathcal V$ has  $\leq m^2 - m +1$  Gumm terms.
Hence (1) follows from (2).   
\end{proof}  

There are versions of Lemma \ref{gt}
and of Theorem \ref{jgt}
obtained using Gumm directed terms \cite{kkmm}
in place of Gumm terms. See \cite{ricm}.

As far as we know,
there is the possibility that Theorem \ref{jgt}
is an empty result, namely that, say, for $m \geq 3$,  every congruence
 distributive variety with Day level $m$  has J{\'o}nsson
level $< 2m^2 -2m -2 $
(for, were this the case, then Corollary \ref{ell} would give a better
evaluation of $ J _{ \mathcal V} (2) $).
However, this would be a quite unexpected result.
What Theorem \ref{jgt} and the above comment
do show is that \emph{in any case} 
  the Day level of a congruence distributive
variety has an influence on the very low levels
of the J{\'o}nsson spectrum.
This connection appears rather surprising,
whichever of the above cases occurs.

\acknowledgement{We thank Antonio Pasini for encouragement
and for providing us with a lot of material
 when we were unexperienced in universal algebra.
We thank Keith Kearnes for useful correspondence.
 We thank the students of Tor Vergata
University for  stimulating discussions.}

\smallskip 

{\scriptsize
Though the author has done his best efforts to compile the following
list of references in the most accurate way,
 he acknowledges that the list might 
turn out to be incomplete
or partially inaccurate, possibly for reasons not depending on him.
It is not intended that each work in the list
has given equally significant contributions to the discipline.
Henceforth the author disagrees with the use of the list
(even in aggregate forms in combination with similar lists)
in order to determine rankings or other indicators of, e.~g., journals, individuals or
institutions. In particular, the author 
 considers that it is highly  inappropriate, 
and strongly discourages, the use 
(even in partial, preliminary or auxiliary forms)
of indicators extracted from the list in decisions about individuals (especially, job opportunities, career progressions etc.), attributions of funds, and selections or evaluations of research projects.
\par
}

\end{document}